\documentclass[9pt,a4paper]{article}
\usepackage{amsmath}
\usepackage{amsfonts}
\usepackage{amsthm}
\usepackage{amssymb}
\usepackage{booktabs,amsmath}
\def\Xint#1{\mathchoice
   {\XXint\displaystyle\textstyle{#1}}%
   {\XXint\textstyle\scriptstyle{#1}}%
   {\XXint\scriptstyle\scriptscriptstyle{#1}}%
   {\XXint\scriptscriptstyle\scriptscriptstyle{#1}}%
   \!\int}
\def\XXint#1#2#3{{\setbox0=\hbox{$#1{#2#3}{\int}$}
     \vcenter{\hbox{$#2#3$}}\kern-.5\wd0}}

\def\dashint{\Xint-}
\usepackage[english]{babel}
\usepackage{xcolor}
\usepackage{enumerate}

\theoremstyle{plain}
\newtheorem{theo}{Theorem}[section]%[Definizioni]
\newtheorem{lem}[theo]{Lemma}%[Definizioni]
\newtheorem{prop}[theo]{Proposition}%[Definizioni]
\newtheorem{cor}[theo]{Corollary}%

\theoremstyle{definition}
\newtheorem{definition}[theo]{Definition}%[Definizioni]
\theoremstyle{remark}
\newtheorem{rem}[theo]{Remark}%[Definizioni]
\numberwithin{equation}{section}

\newcommand{\R}{\mathbb{R}}
\newcommand{\N}{\mathbb{N}}

\title{Doubling Inequality at the Boundary for the Kirchhoff - Love Plate's Equation with Dirichlet Conditions
\thanks{Antonino Morassi is supported by PRIN 2015TTJN95 ``Identificazione
e diagnostica di sistemi strutturali complessi''. Edi Rosset and
Sergio Vessella are supported by Progetto GNAMPA 2019 ``Propriet\`a delle soluzioni di equazioni alle derivate parziali e applicazioni ai problemi inversi'' Istituto Nazionale di Alta Matematica (INdAM). Edi Rosset is
supported by FRA2016 ``Problemi inversi, dalla stabilit\`a alla
ricostruzione'', Universit\`a degli Studi di Trieste.}}

\author{Antonino Morassi\thanks{Dipartimento Politecnico di Ingegneria e Architettura,
Universit\`a degli Studi di Udine, via Cotonificio 114, 33100
Udine, Italy. E-mail: \textsf{antonino.morassi@uniud.it}}, \  Edi
Rosset\thanks{Dipartimento di Matematica e Geoscienze,
Universit\`a degli Studi di Trieste, via Valerio 12/1, 34127
Trieste, Italy. E-mail: \textsf{rossedi@units.it}} \ and Sergio
Vessella\thanks{Dipartimento di Matematica e Informatica ``Ulisse
Dini'', Universit\`a degli Studi di Firenze, Via Morgagni 67/a,
50134 Firenze, Italy. E-mail: \textsf{sergio.vessella@unifi.it}}}

\date{}

\begin{document}

\maketitle

\noindent \textbf{Abstract.} The main result of this paper is a doubling inequality at the boundary for solutions to the Kirchhoff-Love isotropic plate's equation satisfying homogeneous Dirichlet conditions. This result, like the three sphere inequality with optimal exponent at the boundary proved in Alessandrini, Rosset, Vessella, Arch. Ration. Mech. Anal. (2019), implies the Strong Unique Continuation Property at the Boundary (SUCPB). Our approach is based on a suitable Carleman estimate, and involves an ad hoc reflection of the solution. We also give a simple application of our main result, by weakening the standard hypotheses ensuring uniqueness for the Cauchy problem for the plate equation.
\medskip

\medskip


\medskip

\medskip


\section{Introduction} \label{sec:
introduction}

Let us consider the following Kirchhoff - Love plate's equation in a bounded domain
$\Omega\subset \R^2$
\begin{equation}
    \label{eq:equazione_piastra-int}
    L(v) := {\rm div}\left ({\rm div} \left ( B(1-\nu)\nabla^2 v + B\nu \Delta v I_2 \right ) \right )=0, \qquad\hbox{in
    } \Omega,
\end{equation}
where $v$ represents the transversal displacement, $B$ is the \emph{bending stiffness} and $\nu$ the \emph{Poisson's coefficient} (see \eqref{eq:3.stiffness}--\eqref{eq:3.E_nu} for precise definitions).

Assuming $B,\nu\in C^4(\overline{\Omega})$ and given an open portion $\Gamma$ of $\partial\Omega$ of $C^{6, \alpha}$ class, the following Strong Unique Continuation Property at the Boundary (SUCPB) has been proved in \cite{l:ARV}

\begin{equation}\label{formulaz-sucpb-piastra}
\begin{cases}
 Lv=0, \mbox{ in } \Omega, \\
v =\frac{\partial v}{\partial n}=0, \mbox{ on } \Gamma,   \\
\int_{\Omega\cap B_r(P)}v^2=\mathcal{O}(r^k), \mbox{ as } r\rightarrow 0, \forall k\in \mathbb{N},
\end{cases}\Longrightarrow \quad v\equiv 0 \mbox{ in } \Omega,
\end{equation}
where $P$ is any point in $\Gamma$ and $n$ is the outer unit normal.
The above result is the first nontrivial SUCPB result for fourth-order elliptic equations.  Until paper \cite{l:ARV} appeared, such SUCPB results were confined to second order elliptic partial differential equations \cite{l:AdEsKe}, \cite{l:AdEs}, \cite{l:ARRV}, \cite{l:ApEsWaZh}, \cite{l:BaGa}, \cite{l:BoWo}, \cite{l:KeWa}, \cite{l:KuNy}, \cite{l:Si}.

The SUCPB and the related quantitative estimates (in the form of three spheres inequality and doubling inequality), turned out to be a crucial property to prove optimal stability estimates for inverse problems with unknown boundaries for second order elliptic equations \cite{l:AlBeRoVe}. The optimality of the logarithmic character of the stability estimates in \cite{l:AlBeRoVe} has been proved in \cite{l:DcRo}. For this reason, the investigation about the SUCPB has been successfully extended to second order parabolic equations \cite{l:CaRoVe}, \cite{l:DcRoVe}, \cite{l:EsFe}, \cite{l:EsFeVe}, \cite{l:Ve1} and to wave equation with time independent coefficients \cite{l:SiVe}, \cite{l:Ve2}. We refer to the Introduction and the references in
\cite{l:ARV} for a more complete description of the unique continuation principle in the interior for plate equation and for the SUCPB for elliptic equations.

An application of the SUCPB proved in \cite{l:ARV} to inverse problems has been given in \cite{l:mrv19}, where an optimal stability estimate for the  identification of a rigid inclusion in an isotropic Kirchhoff - Love plate was proved. A crucial tool used in \cite{l:mrv19} is a three spheres inequality at the boundary with optimal exponent \cite[Theorem 5.1]{l:ARV}.

 The main result of the present paper is the following doubling inequality at the boundary (see Theorem \ref{theo:40.teo} for precise statement)

 \begin{equation}
    \label{eq:10.6.1102-intro}
    \int_{B_{2r}(P)\cap \Omega}|v|^2\leq K \int_{B_{r}(P)\cap \Omega}|v|^2,
\end{equation}
where $K$ is constant depending by $v$, but \textit{independent} of $r$. It is well known that also doubling inequality implies the SUCPB, \cite{l:Giaq}, \cite{l:GaLi}.
The interior version of the doubling inequality for the plate equation was obtained in \cite{l:LiNaWa} and \cite{l:DLMRVW} for anisotropic plates. It is worth noticing that the doubling inequality turns out to be a more powerful tool than three spheres inequality. In fact, the doubling inequality in the interior has been employed to investigate the smallness propagation {}from measurable sets (of positive measure) of a solution to second order elliptic equation \cite{l:LoMa}, and to prove size estimates for general inclusions in electric conductors and in elastic bodies \cite{l:AMR}, \cite{l:DLMRVW}, \cite{l:DSV}, \cite{l:mrv09}.
In particular, in Corollary \ref{application} we give a first simple application of the doubling inequality at the boundary \eqref{eq:10.6.1102-intro}, which allows to weaken the hypotheses ensuring uniqueness for the Cauchy problem for Kirchhoff - Love isotropic plates.

The proof of inequality \eqref{eq:10.6.1102-intro} is based on a strategy similar but sharper than the one followed in \cite{l:ARV}. Firstly, similarly to \cite{l:ARV}, we flatten the boundary $\Gamma$ by introducing a suitable conformal mapping (see Proposition \ref{prop:conf_map}). Then we combine a reflection argument with the following Carleman estimate

\begin{gather}
    \label{eq:24.4-intro}
	\tau^4r^2\int\rho^{-2-2\tau}|U|^2dxdy +\sum_{k=0}^3 \tau^{6-2k}\int\rho^{2k+1-2\tau}|D^kU|^2dxdy\leq \\\nonumber  \leq C
	\int\rho^{8-2\tau}(\Delta^2 U)^2dxdy,
\end{gather}
for every $\tau\geq \overline{\tau}$, for every $r\in (0, 1)$ and for every $U\in C^\infty_0(B_{1}\setminus\ \overline{B}_{r/4})$, where $\rho(x,y)\sim \sqrt{x^2+y^2}$ as $(x,y)\rightarrow (0,0)$ (see Proposition \ref{prop:Carleman} for a precise statement). We emphasize that, with respect to the Carleman estimate employed in \cite{l:ARV}, the presence of the first term in the left hand side of \eqref{eq:24.4-intro} is the key ingredient in order to prove our doubling inequality at the boundary. At the best of our knowledge, Bakri is
the first author who derived a doubling inequality in the interior
starting {}from a Carleman estimate of the kind \eqref{eq:24.4-intro} \cite{l:Bakri1}, see also \cite{l:Bakri2} and \cite{l:Zhu}.

The paper is organized as follows. In Section \ref{sec:
notation} we introduce some notation and definitions, and state our main result, Theorem \ref{theo:40.teo}. In Section \ref{sec:Preliminary} we collect some auxiliary propositions, precisely Proposition \ref{prop:conf_map} introducing the conformal map used to flatten the boundary, Propositions \ref{prop:16.1} and \ref{prop:19.2} concerning the reflection with respect to flat boundaries and its properties, a Hardy's inequality (Proposition \ref{prop:Hardy}), the Carleman estimate for bi-Laplace operator (Proposition \ref{prop:Carleman}), and some interpolation estimates (Lemma \ref{lem:Agmon}) and Caccioppoli-type inequality (Lemma \ref{lem:intermezzo}).
In Section \ref{sec:doubling} we establish the doubling inequality at the boundary, and we state and prove Corollary \ref{application}.
Finally, the Appendix contains the proof of Proposition \ref{prop:Carleman}, in which we have presented the arguments in detailed form for the reader's convenience.

\section{Notation and main result} \label{sec:
notation}

We shall generally denote points in $\R^2$ by $x=(x_1,x_2)$ or $y=(y_1,y_2)$, except for
Sections \ref{sec:Preliminary} and \ref{sec:doubling} where we rename $x,y$ the coordinates in $\R^2$.
In places we will use equivalently the symbols $D$ and $\nabla$ to denote the gradient of a function. Also we use the multi-index notation.
We shall denote by $B_r(P)$ the disc in $\R^2$ of radius $r$ and
center $P$, by $B_r$ the disk of radius $r$ and
center $O$, by $B_r^+$, $B_r^-$ the hemidiscs in $\R^2$  of radius $r$ and
center $O$ contained in the halfplanes $\R^2_+= \{x_2>0\}$, $\R^2_-= \{x_2<0\}$ respectively, and by $R_{
a,b}$ the rectangle $(-a,a)\times(-b,b)$.

Given a matrix $A =(a_{ij})$, we shall denote by $|A|$ its Frobenius norm $|A|=\sqrt{\sum_{i,j}a_{ij}^2}$.

Along our proofs, we shall denote by $C$ a constant which may change {}from line to line.

\begin{definition}
  \label{def:reg_bordo} (${C}^{k,\alpha}$ regularity)
Let $\Omega$ be a bounded domain in ${\R}^{2}$. Given $k,\alpha$,
with $k\in\N$, $0<\alpha\leq 1$, we say that a portion $S$ of
$\partial \Omega$ is of \textit{class ${C}^{k,\alpha}$ with
constants $r_{0}$, $M_{0}>0$}, if, for any $P \in S$, there
exists a rigid transformation of coordinates under which we have
$P=0$ and
\begin{equation*}
  \Omega \cap R_{r_0,2M_0r_0}=\{x \in R_{r_0,2M_0r_0} \quad | \quad
x_{2}>g(x_1)
  \},
\end{equation*}
where $g$ is a ${C}^{k,\alpha}$ function on
$[-r_0,r_0]$
satisfying
\begin{equation*}
g(0)=g'(0)=0,
\end{equation*}
\begin{equation*}
\|g\|_{{C}^{k,\alpha}([-r_0,r_0])} \leq M_0r_0,
\end{equation*}
where
\begin{equation*}
\|g\|_{{C}^{k,\alpha}([-r_0,r_0])} = \sum_{i=0}^k  r_0^i\sup_{[-r_0,r_0]}|g^{(i)}|+r_0^{k+\alpha}|g|_{k,\alpha},
\end{equation*}

\begin{equation*}
|g|_{k,\alpha}= \sup_ {\overset{\scriptstyle t,s\in [-r_0,r_0]}{\scriptstyle
t\neq s}}\left\{\frac{|g^{(k)}(t) - g^{(k)}(s)|}{|t-s|^\alpha}\right\}.
\end{equation*}

\end{definition}

We shall consider an isotropic thin elastic plate $\Omega\times \left[-\frac{h}{2},\frac{h}{2}\right]$, having middle plane $\Omega$ and thickness $h$. Under the Kirchhoff - Love theory, the transversal displacement $v$ satisfies the following fourth-order partial differential equation

\begin{equation}
    \label{eq:equazione_piastra}
    L(v) := {\rm div}\left ({\rm div} \left ( B(1-\nu)\nabla^2 v + B\nu \Delta v I_2 \right ) \right )=0, \qquad\hbox{in
    } \Omega.
\end{equation}

Here the \emph{bending stiffness} $B$ is given by

\begin{equation}
  \label{eq:3.stiffness}
  B(x)=\frac{h^3}{12}\left(\frac{E(x)}{1-\nu^2(x)}\right),
\end{equation}
and the \emph{Young's modulus} $E$ and the \emph{Poisson's coefficient} $\nu$ can be written in terms of the Lam\'{e} moduli as follows
\begin{equation}
  \label{eq:3.E_nu}
  E(x)=\frac{\mu(x)(2\mu(x)+3\lambda(x))}{\mu(x)+\lambda(x)},\qquad\nu(x)=\frac{\lambda(x)}{2(\mu(x)+\lambda(x))}.
\end{equation}

On the Lam\'{e} moduli, we shall assume

\medskip
\emph{i) Strong convexity}:

\begin{equation}
  \label{eq:3.Lame_convex}
  \mu(x)\geq \alpha_0>0,\qquad 2\mu(x)+3\lambda(x)\geq\gamma_0>0, \qquad \hbox{ in } \Omega,
\end{equation}
where $\alpha_0$, $\gamma_0$ are positive constants;

\medskip
\emph{ii) Regularity}:

\begin{equation}
  \label{eq:C4Lame}
  \|\lambda\|_{C^4(\overline{\Omega}_{r_0})}, \|\mu\|_{C^4(\overline{\Omega}_{r_0})}\leq \Lambda_0,
\end{equation}
with $\Lambda_0$ a positive constant.

It is easy to see that equation \eqref{eq:equazione_piastra} can be rewritten in the form

\begin{equation}
    \label{eq:equazione_piastra_non_div}
    \Delta^2 v= \widetilde{a}\cdot \nabla\Delta v + \widetilde{q}_2(v) \qquad\hbox{in
    } \Omega,
\end{equation}
with
\begin{equation}
    \label{eq:vettore_a_tilde}
    \widetilde{a}=-2\frac{\nabla B}{B},
\end{equation}

\begin{equation}
    \label{eq:q_2}
    \widetilde{q}_2(v)=-\sum_{i,j=1}^2\frac{1}{B}\partial^2_{ij}(B(1-\nu)+\nu B\delta_{ij})\partial^2_{ij} v.
\end{equation}
Let
\begin{equation}
   \label{eq:Omega_r_0}
\Omega_{r_0} = \left\{ x\in R_{r_0,2M_0r_0}\ |\ x_2>g(x_1) \right\},
\end{equation}

\begin{equation}
   \label{eq:Gamma_r_0}
\Gamma_{r_0} = \left\{(x_1,g(x_1))\ |\ x_1\in (-r_0,r_0)\right\},
\end{equation}
with
\begin{equation*}
g(0)=g'(0)=0,
\end{equation*}

\begin{equation}
   \label{eq:regol_g}
\|g\|_{{C}^{6,\alpha}([-r_0,r_0])} \leq M_0r_0,
\end{equation}
for some $\alpha\in (0,1]$.
Let
$v\in H^2(\Omega_{r_0})$ satisfy
\begin{equation}
  \label{eq:equat_u_tilde}
  L(v)= 0, \quad \hbox{ in } \Omega_{r_0},
\end{equation}

\begin{equation}
  \label{eq:Diric_u_tilde}
  v =  \frac{\partial v}{\partial n}= 0, \quad \hbox{ on } \Gamma_{r_0},
\end{equation}
where $L$ is given by \eqref{eq:equazione_piastra} and $n$ denotes the outer unit normal.

The assumptions \eqref{eq:3.Lame_convex}, \eqref{eq:regol_g} and \eqref{eq:C4Lame} guarantee that $v\in H^6(\Omega_r)$, see for instance \cite{l:a65}.

\begin{theo} [{\bf Doubling inequality at the boundary}]
    \label{theo:40.teo}
    Under the above hypotheses, there exist $k>1$ and $C>1$ only depending on $\alpha_0$, $\gamma_0$, $\Lambda_0$, $M_0$, $\alpha$, such that, for every $r<\frac{r_0}{C}$ we have
		\begin{equation}
    \label{eq:10.6.1102}
    \int_{B_{2r}\cap \Omega_{r_0}}|v|^2\leq CN^k\int_{B_{r}\cap \Omega_{r_0}}|v|^2,
\end{equation}
where
\begin{equation}
    \label{eq:10.6.1108}
    N=\frac{\int_{B_{r_0}\cap \Omega_{r_0}}|v|^2}{\int_{B_{\frac{r_0}{C}}\cap \Omega_{r_0}}|v|^2}.
\end{equation}

\end{theo}

\section{Preliminary results} \label{sec:Preliminary}

In the following Proposition, proved in \cite{l:ARV}, we introduce a conformal map which flattens the boundary
$\Gamma_{r_0}$ and preserves the structure of equation \eqref{eq:equazione_piastra_non_div}.

\begin{prop} [{\bf Reduction to a flat boundary}]
    \label{prop:conf_map}
Under the hypotheses of Theorem \ref{theo:40.teo}, there exists an injective sense preserving differentiable map
\begin{equation*}
\Phi=(\varphi,\psi):[-1,1]\times[0,1]\rightarrow \overline{\Omega}_{r_0}
\end{equation*}
which is conformal and satisfies
\begin{equation}
   \label{eq:9.assente}
  \Phi((-1,1)\times(0,1))\supset B_{\frac{r_0}{K}}(0)\cap \Omega_{r_0},
\end{equation}
\begin{equation}
   \label{eq:9.2b}
  \Phi(([-1,1]\times\{0\})= \left\{ (x_1,g(x_1))\ |\ x_1\in [-r_1,r_1]\right\},
\end{equation}

\begin{equation}
   \label{eq:9.2a}
  \Phi(0,0)= (0,0),
\end{equation}

\begin{equation}
    \label{eq:gradPhi}
  \frac{c_0r_0}{2C_0}\leq |D\Phi(y)|\leq \frac{r_0}{2}, \quad \forall y\in [-1,1]\times[0,1],
\end{equation}
\begin{equation}
   \label{eq:gradPhiInv}
\frac{4}{r_0}\leq |D\Phi^{-1}(x)|\leq \frac{4C_0}{c_0r_0}, \quad\forall x\in \Phi([-1,1]\times[0,1]),
\end{equation}
\begin{equation}
   \label{eq:stimaPhi}
\frac{r_0}{K}|y|\leq|\Phi(y)|\leq \frac{r_0}{2}|y|, \quad \forall y\in [-1,1]\times[0,1],
\end{equation}
with
$K>8$, $0<c_0<C_0$ being constants only depending on $M_0$ and $\alpha$.

Letting
\begin{equation}
  \label{eq:def_sol_composta}
  u(y) = v(\Phi(y)), \quad y\in [-1,1]\times[0,1],
\end{equation}
then $u\in H^6((-1,1)\times(0,1))$ and satisfies
\begin{equation}
    \label{eq:equazione_sol_composta}
    \Delta^2 u= a\cdot \nabla\Delta u + q_2(u), \qquad\hbox{in
    } (-1,1)\times(0,1),
\end{equation}
\begin{equation}
    \label{eq:Dirichlet_sol_composta}
    u(y_1,0)= u_{y_2}(y_1,0) =0, \quad \forall y_1\in (-1,1),
\end{equation}
where
\begin{equation*}
  a(y) = |\nabla \varphi(y)|^2\left([D\Phi(y)]^{-1}\widetilde{a}(\Phi(y))-2\nabla(|\nabla \varphi(y)|^{-2})\right),
\end{equation*}
$a\in C^3([-1,1]\times[0,1], \R^2)$, $q_2=\sum_{|\alpha|\leq 2}c_\alpha D^\alpha$ is a second order elliptic operator with coefficients $c_\alpha\in C^2([-1,1]\times[0,1])$,
 satisfying
\begin{equation}
    \label{eq:15.2}
    \|a\|_{ C^3([-1,1]\times[0,1], \R^2)}\leq M_1,\quad \|c_\alpha\|_{ C^2([-1,1]\times[0,1])}\leq M_1,
\end{equation}
with $M_1>0$ only depending on $\alpha_0, \gamma_0, \Lambda_0, M_0, \alpha$.
\end{prop}

In order to simplify the notation, in the sequel of this section we rename $x,y$ the coordinates in $\R^2$.

Let $u\in H^6(B_1^+)$ be a solution to
\begin{equation}
    \label{eq:15.1a}
    \Delta^2 u= a\cdot \nabla\Delta u + q_2(u), \qquad\hbox{in
    } B_1^+,
\end{equation}
\begin{equation}
    \label{eq:15.1b}
    u(x,0)=u_y(x,0) =0, \quad \forall x\in (-1,1),
\end{equation}
with $q_2=\sum_{|\alpha|\leq 2}c_\alpha D^\alpha$,
\begin{equation}
    \label{eq:15.2_bis}
    \|a\|_{ C^3(\overline{B}_1^+, \R^2)}\leq M_1,\quad \|c_\alpha\|_{ C^2(\overline{B}_1^+)}\leq M_1,
\end{equation}
for some positive constant $M_1$.

Let us define the following extension of $u$ to $B_1$ (see \cite{l:Jo})
\begin{equation}
    \label{eq:16.1}
    \overline{u}(x,y)=\left\{
  \begin{array}{cc}
    u(x,y), & \hbox{ in } B_1^+,\\
    w(x,y), & \hbox{ in } B_1^-,
  \end{array}
\right.
\end{equation}
where
\begin{equation}
    \label{eq:16.2}
    w(x,y)= -[u(x,-y)+2yu_y(x,-y)+y^2\Delta u(x,-y)].
\end{equation}

We refer to \cite{l:ARV} for a proof of Propositions \ref{prop:16.1} and \ref{prop:19.2} below.

\begin{prop}
    \label{prop:16.1}
Let
\begin{equation}
    \label{eq:16.3}
F:=a\cdot \nabla\Delta u + q_2(u).
\end{equation}
Then $F\in H^2(B_1^+)$, $\overline{u}\in H^4(B_1)$,
\begin{equation}
    \label{eq:16.4}
\Delta^2 \overline{u} = \overline{F},\quad \hbox{ in } B_1,
\end{equation}
where
\begin{equation}
    \label{eq:16.5}
    \overline{F}(x,y)=\left\{
  \begin{array}{cc}
    F(x,y), & \hbox{ in } B_1^+,\\
    F_1(x,y), & \hbox{ in } B_1^-,
  \end{array}
\right.
\end{equation}
and
\begin{equation}
    \label{eq:16.6}
    F_1(x,y)= -[5F(x,-y)-6yF_y(x,-y)+y^2\Delta F(x,-y)].
\end{equation}
\end{prop}

In the following proposition, we shall denote by $P_k$, for $k=2,3$, any differential operator of the form
$\sum_{|\alpha|\leq k}c_\alpha(\cdot)D^\alpha$,
with $\|c_\alpha\|_{L^\infty}\leq cM_1$, where $c$ is an absolute constant.

\begin{prop}
    \label{prop:19.2}
		For every $(x,y)\in B_1^-$, we have
\begin{equation}
    \label{eq:19.1}
    F_1(x,y)= H(x,y)+(P_2(w))(x,y)+(P_3(u))(x,-y),
\end{equation}
where
\begin{multline}
    \label{eq:19.2}
    H(x,y)= 6\frac{a_1}{y}(w_{yx}(x,y)+u_{yx}(x,-y))+\\
		+6\frac{a_2}{y}(-w_{yy}(x,y)+u_{yy}(x,-y))
		-\frac{12a_2}{y}u_{xx}(x,-y),
\end{multline}
where $a_1,a_2$ are the components of the vector $a$.
Moreover, for every $x\in (-1,1)$,
\begin{equation}
    \label{eq:23.1}
    w_{yx}(x,0)+u_{yx}(x,0)=0,
\end{equation}
\begin{equation}
    \label{eq:23.2}
    -w_{yy}(x,0)+u_{yy}(x,0)=0,
\end{equation}
\begin{equation}
    \label{eq:23.3}
    u_{xx}(x,0)=0.
\end{equation}
\end{prop}

We shall also use the following Hardy's inequality (\cite[\S 7.3, p. 175]{l:HLP34}), for a proof see also \cite{l:T67}.

\begin{prop} [{\bf Hardy's inequality}]
    \label{prop:Hardy}
		Let $f$ be an absolutely continuous function defined in $[0,+\infty)$, such that
		$f(0)=0$. Then
\begin{equation}
    \label{eq:24.1}
		\int_0^{+\infty} \frac{f^2(s)}{s^2}ds\leq 4 \int_0^{+\infty} (f'(s))^2ds.
\end{equation}
\end{prop}

Another basic ingredient for our proof of the doubling inequality at the boundary is the following Carleman estimate, whose proof is postponed in the Appendix.

\begin{prop} [{\bf Carleman estimate}]
    \label{prop:Carleman}
		Let us define
\begin{equation}
    \label{eq:24.2}
		\rho(x,y) = \phi\left(\sqrt{x^2+y^2}\right),
\end{equation}
where
\begin{equation}
    \label{eq:24.3}
		\phi(s) = \frac{s}{\left(1+\sqrt{s}\right)^2}.
\end{equation}
Then there exist absolute constants $\overline{\tau}>1$, $C>1$ such that
\begin{gather}
    \label{eq:24.4}
	\tau^4r^2\int\rho^{-2-2\tau}|U|^2dxdy +\sum_{k=0}^3 \tau^{6-2k}\int\rho^{2k+1-2\tau}|D^kU|^2dxdy \\\nonumber  \leq C
	\int\rho^{8-2\tau}(\Delta^2 U)^2dxdy,
\end{gather}
for every $\tau\geq \overline{\tau}$, for every $r\in (0, 1)$ and for every $U\in C^\infty_0(B_{1}\setminus\ \overline{B}_{r/4})$.
\end{prop}

\begin{rem}
   \label{rem:stima_rho}
	Let us notice that
	\begin{equation*}
\frac{s}{4}\leq \varphi(s)\leq s, \quad \forall s\leq 1,
\end{equation*}
\begin{equation}
    \label{eq:stima_rho}
	\frac{\sqrt{x^2+y^2}}{4}\leq \rho(x,y)\leq \sqrt{x^2+y^2}, \quad \forall (x,y)\in B_1.
\end{equation}	
	
\end{rem}

We shall need also the following results.

\begin{lem} [Interpolation estimates]
    \label{lem:Agmon}
Let $0<\epsilon\leq 1$ and $m\in \N$, $m\geq 2$. For any $j=1,\cdots, m-1$ there exists an absolute constant
$C_{m,j}$ such that for every $v\in H^m(B_r^+)$,
\begin{equation}
    \label{eq:3a.2}
	r^j\|D^jv\|_{L^2(B_r^+)}\leq C_{m,j}\left(\epsilon r^m\|D^mv\|_{L^2(B_r^+)}
	+\epsilon^{-\frac{j}{m-j}}\|v\|_{L^2(B_r^+)}\right).
\end{equation}	
\end{lem}

See \cite[Theorem 3.3]{l:a65}.

\begin{lem}[Caccioppoli-type inequality]
    \label{lem:intermezzo}
		Let $u\in H^6(B_1^+)$ be a solution to \eqref{eq:15.1a}--\eqref{eq:15.1b}, with $a$ and $q_2$ satisfying
		\eqref{eq:15.2_bis}. For every $r$, $0<r<1$, we have
		\begin{equation}
    \label{eq:12a.2}
	\|D^hu\|_{L^2(B_{\frac{r}{2}}^+)}\leq \frac{C}{r^h}\|u\|_{L^2(B_r^+)}, \quad \forall
	h=1, ..., 6,
\end{equation}	
where $C$ is a constant only depending on $\alpha_0$, $\gamma_0$ and $\Lambda_0$.
\end{lem}
See \cite[Lemma 4.7]{l:ARV}.

\section{Proof of the main theorem} \label{sec:doubling}

\begin{lem}\label{newlemma1}
Let $u\in H^6(B_1^+)$ be a solution to \eqref{eq:15.1a}--\eqref{eq:15.1b}. There exists a positive number $\overline{R}_0\in (0,1)$, depending on $M_1$ only, such that, for every $R$ and for every $r$ such that $0<2r<R<\frac{\overline{R}_0}{2}$, we have
\begin{gather}
    \label{eq:37.1bis}
    R(2r)^{-2\tau}\int_{B^+_{2r}}|u|^2+R^{1-2\tau}\int_{B_{R}^+}
	|u|^2\leq\\ \nonumber
	\leq C (M_1^2+1)\left[\left(\frac{r}{4}\right)^{-2\tau}\int_{B_r^+}|u|^2+
	\left(\frac{\overline{R}_0}{2}\right)^{-2\tau}\int_{B_{\overline{R}_0}^+}|u|^2
	\right],
\end{gather}
for every $\tau\geq \widetilde{\tau}$, with $\widetilde{\tau}, C$ positive
absolute constants.
\end{lem}

\begin{proof}
Let $R_0\in (0,1)$ to be chosen later and let
\begin{equation}
    \label{eq:25.1}
0<r<R<\frac{R_0}{2}.
\end{equation}
Let $\eta\in C^\infty_0((0,1))$ such that
\begin{equation}
    \label{eq:25.2}
    0\leq \eta\leq 1
\end{equation}
\begin{equation}
    \label{eq:25.4}
\eta=0, \quad \hbox{ in }\left(0,\frac{r}{4}\right)\cup \left(\frac{2}{3}R_0,1\right), \quad \eta=1, \quad \hbox{ in }\left[\frac{r}{2}, \frac{R_0}{2}\right],
\end{equation}
\begin{equation}
    \label{eq:25.6}
\left|\frac{d^k\eta}{dt^k}(t)\right|\leq C r^{-k}, \quad \hbox{ in }\left(\frac{r}{4}, \frac{r}{2}\right),\quad\hbox{ for } 0\leq k\leq 4,
\end{equation}
\begin{equation}
    \label{eq:25.7}
\left|\frac{d^k\eta}{dt^k}(t)\right|\leq C R_0^{-k}, \quad \hbox{ in }\left(\frac{R_0}{2}, \frac{2}{3}R_0\right),\quad\hbox{ for } 0\leq k\leq 4.
\end{equation}
Let us define
\begin{equation}
    \label{eq:25.5}
\xi(x,y)=\eta(\sqrt{x^2+y^2}).
\end{equation}
By a density argument, we may apply the Carleman estimate \eqref{eq:24.4} to $U=\xi \overline{u}$, where $\overline{u}$ has been defined in \eqref{eq:16.1}, obtaining
\begin{gather}
    \label{eq:26.1}
	\tau^4r^2\int_{B_{R_0}}\rho^{-2-2\tau}|\xi \overline{u}|^2+\\\nonumber
    +\sum_{k=0}^3 \tau^{6-2k}\int_{B_{R_0}^+}\rho^{2k+1-2\tau}|D^k(\xi u)|^2
	+\sum_{k=0}^3 \tau^{6-2k}\int_{B_{R_0}^-}\rho^{2k+1-2\tau}|D^k(\xi w)|^2\leq \\\nonumber
	\leq C
	\int_{B_{R_0}^+}\rho^{8-2\tau}|\Delta^2(\xi u)|^2+
	C\int_{B_{R_0}^-}\rho^{8-2\tau}|\Delta^2(\xi w)|^2,
\end{gather}
for $\tau\geq \overline{\tau}$ and $C$ an absolute constant.

Let us set
\begin{multline}
    \label{eq:27.1}
	J_0 =\int_{B_{r/2}^+\setminus B_{r/4}^+}\rho^{8-2\tau}
	\sum_{k=0}^3 (r^{k-4}|D^k u|)^2+
	\int_{B_{r/2}^-\setminus B_{r/4}^-}\rho^{8-2\tau}
	\sum_{k=0}^3 (r^{k-4}|D^k w|)^2,
\end{multline}
\begin{multline}
    \label{eq:27.2}
	J_1 =\int_{B_{2R_0/3}^+\setminus B_{R_0/2}^+}\rho^{8-2\tau}
	\sum_{k=0}^3 (R_0^{k-4}|D^k u|)^2+
	\int_{B_{2R_0/3}^-\setminus B_{R_0/2}^-}\rho^{8-2\tau}
	\sum_{k=0}^3 (R_0^{k-4}|D^k w|)^2.
\end{multline}

By \eqref{eq:25.2}--\eqref{eq:27.2} we have
\begin{gather}
    \label{eq:27.3}
    \tau^4r^2\int_{B_{R_0}}\rho^{-2-2\tau}|\xi \overline{u}|^2+\\ \nonumber
    \sum_{k=0}^3 \tau^{6-2k}\int_{B_{R_0}^+}\rho^{2k+1-2\tau}|D^k(\xi u)|^2
	+\sum_{k=0}^3 \tau^{6-2k}\int_{B_{R_0}^-}\rho^{2k+1-2\tau}|D^k(\xi w)|^2\leq \\ \nonumber
     \leq C
	\int_{B_{R_0}^+}\rho^{8-2\tau}\xi^2|\Delta^2 u|^2+
	C\int_{B_{R_0}^-}\rho^{8-2\tau}\xi^2|\Delta^2 w|^2+CJ_0+CJ_1,
\end{gather}
for $\tau\geq \overline{\tau}$, with $C$ an absolute constant.

By \eqref{eq:15.1a} and \eqref{eq:15.2_bis} we have
\begin{equation}
    \label{eq:28.1}
	\int_{B_{R_0}^+}\rho^{8-2\tau}\xi^2|\Delta^2 u|^2\leq
	CM_1^2\int_{B_{R_0}^+}\rho^{8-2\tau}\xi^2\sum_{k=0}^3|D^k u|^2.
\end{equation}

By \eqref{eq:16.4}, \eqref{eq:16.6} and by making the change of variables
$(x,y)\rightarrow(x,-y)$ in the integrals involving the function $u(x,-y)$,
we can estimate the second term in the right hand side of \eqref{eq:27.3} as follows
\begin{multline}
    \label{eq:28.2}
	\int_{B_{R_0}^-}\rho^{8-2\tau}\xi^2|\Delta^2 w|^2\leq
	C\int_{B_{R_0}^-}\rho^{8-2\tau}\xi^2|H(x,y)|^2+\\
	+CM_1^2\int_{B_{R_0}^-}\rho^{8-2\tau}\xi^2\sum_{k=0}^2|D^k w|^2+
	CM_1^2\int_{B_{R_0}^+}\rho^{8-2\tau}\xi^2\sum_{k=0}^3|D^k u|^2.
\end{multline}
Now, let us split the integral in the right hand side of \eqref{eq:28.1} and the second and third integrals in the right hand side of \eqref{eq:28.2} over the domains of integration $B_{r/2}^\pm\setminus B_{r/4}^\pm$, $B_{R_0/2}^\pm\setminus B_{r/2}^\pm$, $B_{2R_0/3}^\pm\setminus B_{R_0/2}^\pm$. Then let us insert \eqref{eq:28.1}--\eqref{eq:28.2}in \eqref{eq:27.3}, obtaining
\begin{gather}
    \label{eq:28.4}
     \tau^4r^2\int_{B_{R_0}}\rho^{-2-2\tau}|\xi \overline{u}|^2+\\\nonumber	
     +\sum_{k=0}^3 \tau^{6-2k}\int_{B_{R_0}^+}\rho^{2k+1-2\tau}|D^k(\xi u)|^2
	+\sum_{k=0}^3 \tau^{6-2k}\int_{B_{R_0}^-}\rho^{2k+1-2\tau}|D^k(\xi w)|^2\leq \\ \nonumber
	\leq C\int_{B_{R_0}^-}\rho^{8-2\tau}\xi^2|H(x,y)|^2
	+CM_1^2\int_{B_{R_0/2}^-\setminus B_{r/2}^-}\rho^{8-2\tau}\sum_{k=0}^2|D^k w|^2+\\ \nonumber
	+CM_1^2\int_{B_{R_0/2}^+ \setminus B_{r/2}^+}\rho^{8-2\tau}\sum_{k=0}^3|D^k u|^2
	+C\overline{M}^2_1(J_0+J_1),
\end{gather}
for $\tau\geq \overline{\tau}$, with $C$ an absolute constant, where $\overline{M}_1=\sqrt{M_1^2+1}$.

The second and third integral on the right hand side of \eqref{eq:28.4} can be absorbed by the left hand side so that, by easy calculation, by \eqref{eq:stima_rho} and for
$R_0\leq R_1:=\min\{1,2(2CM_1^2)^{-1}\}$, we have

\begin{gather}
    \label{eq:30.3}
    \tau^4r^2\int_{B_{R_0}}\rho^{-2-2\tau}|\xi \overline{u}|^2+\sum_{k=0}^3 \tau^{6-2k}\int_{B_{R_0/2}^+ \setminus B_{r/2}^+}
	\rho^{2k+1-2\tau}|D^k u|^2+\\ \nonumber
	+\sum_{k=0}^3 \tau^{6-2k}\int_{B_{R_0/2}^- \setminus B_{r/2}^-}
	\rho^{2k+1-2\tau}|D^k w|^2
	\leq \\  \nonumber
	\leq
	C\int_{B_{R_0}^-}\rho^{8-2\tau}\xi^2|H(x,y)|^2
	+C\overline{M}^2_1(J_0+J_1),
\end{gather}
for $\tau\geq \overline{\tau}$, with $C$ an absolute constant.
The first integral on the right hand side can be estimated by proceeding as in \cite[Theorem 5.1]{l:ARV}. For completeness we summarize such an estimate.

By \eqref{eq:19.2} and \eqref{eq:15.2_bis}, we have that
\begin{equation}
    \label{eq:30.4}
	\int_{B_{R_0}^-}\rho^{8-2\tau}\xi^2|H(x,y)|^2\leq CM_1^2(I_1+I_2+I_3),
\end{equation}
with
\begin{equation}
    \label{eq:31.0.4}
	I_1=\int_{-R_0}^{R_0}\left(\int_{-\infty}^0\left|y^{-1}
	u_{xx}(x,-y)\rho^{4-\tau}\xi\right|^2dy\right)dx,
\end{equation}

\begin{equation}
    \label{eq:31.0.1}
	I_2=\int_{-R_0}^{R_0}\left(\int_{-\infty}^0\left|y^{-1}(w_{yy}(x,y)-
	(u_{yy}(x,-y))\rho^{4-\tau}\xi\right|^2dy\right)dx,
\end{equation}
\begin{equation}
    \label{eq:31.0.2}
	I_3=\int_{-R_0}^{R_0}\left(\int_{-\infty}^0\left|y^{-1}(w_{yx}(x,y)+
	(u_{yx}(x,-y))\rho^{4-\tau}\xi\right|^2dy\right)dx.
\end{equation}
Now, let us see that, for $j=1,2,3$,
\begin{multline}
    \label{eq:31.1}
	I_j\leq
	C\int_{B_{R_0}^-}\rho^{8-2\tau}\xi^2|D^3 w|^2
	+C\tau^2\int_{B_{R_0}^-}\rho^{6-2\tau}\xi^2|D^2 w|^2+\\
	+C\int_{B_{R_0}^+}\rho^{8-2\tau}\xi^2|D^3 u|^2
	+C\tau^2\int_{B_{R_0}^+}\rho^{6-2\tau}\xi^2|D^2 u|^2
	+C(J_0+J_1),
\end{multline}
for $\tau\geq \overline{\tau}$, with $C$ an absolute constant.

Let us verify \eqref{eq:31.1} for $j=1$.

By \eqref{eq:23.3} and Hardy's inequality \eqref{eq:24.1} we get
\begin{gather}
    \label{eq:32.2}
	\int_{-\infty}^0\left|y^{-1}
	u_{xx}(x,-y)\rho^{4-\tau}\xi\right|^2dy\leq 4\int_{-\infty}^0\left|\partial_y\left[u_{xx}(x,-y)\rho^{4-\tau}\xi\right]\right|^2dy\leq\\ \nonumber
	\leq 16 \int_{-\infty}^0|u_{xxy}(x,-y)|^2\rho^{8-2\tau}\xi^2dy+16 \int_{-\infty}^0|u_{xx}(x,-y)|^2\left|\partial_y\left(\rho^{4-\tau}\xi\right)\right|^2dy.
\end{gather}
Noticing that $|\rho_y|\leq 1,$ we obtain
\begin{equation}
    \label{eq:32.3}
	\left|\partial_y\left(\rho^{4-\tau}\xi\right)\right|^2\leq 2\xi_y^2\rho^{8-2\tau}+2\tau^2\rho^{6-2\tau}\xi^2,
\end{equation}
for $\tau\geq \widetilde{\tau}:= \max\{\overline{\tau},3\}$.

By integrating over $(-R_0,R_0)$ and by making the change of variables
$(x,y)\rightarrow(x,-y)$, the use of \eqref{eq:32.3} in \eqref{eq:32.2} gives
\begin{gather}
    \label{eq:33.0}
	I_1\leq C\int_{B_{R_0}^+}\xi^2\rho^{8-2\tau}|u_{xxy}|^2
	+C\int_{B_{R_0}^+}\xi_y^2\rho^{8-2\tau}|u_{xx}|^2+
	C\tau^2\int_{B_{R_0}^+}\xi^2\rho^{6-2\tau}|u_{xx}|^2.
\end{gather}
Recalling \eqref{eq:25.2}--\eqref{eq:25.5}, we find \eqref{eq:31.1} for $j=1$, the other cases following by using similar arguments.

Next, by \eqref{eq:30.3}, \eqref{eq:30.4} and \eqref{eq:31.1}, we have
\begin{gather}
    \label{eq:33.1}
    \tau^4r^2\int_{B_{R_0}}\rho^{-2-2\tau}|\xi \overline{u}|^2+\sum_{k=0}^3 \tau^{6-2k}\int_{B_{R_0/2}^+ \setminus B_{r/2}^+}
	\rho^{2k+1-2\tau}|D^k u|^2+\\ \nonumber
	+\sum_{k=0}^3 \tau^{6-2k}\int_{B_{R_0/2}^- \setminus B_{r/2}^-}
	\rho^{2k+1-2\tau}|D^k w|^2
	\leq \\ \nonumber
	\leq
	CM_1^2\int_{B_{R_0}^+}\rho^{8-2\tau}\xi^2|D^3u|^2+
	CM_1^2\int_{B_{R_0}^-}\rho^{8-2\tau}\xi^2|D^3w|^2+\\ \nonumber
	+CM_1^2\tau^2\int_{B_{R_0}^+}\rho^{6-2\tau}\xi^2|D^2u|^2
	+CM_1^2\tau^2\int_{B_{R_0}^-}\rho^{6-2\tau}\xi^2|D^2w|^2
	+C\overline{M_1}^2(J_0+J_1),
\end{gather}
for $\tau\geq \widetilde{\tau}$, with $C$ an absolute constant.

As before, we split the first four integrals in the right hand side of \eqref{eq:33.1}  over the domains of integration $B_{r/2}^\pm\setminus B_{r/4}^\pm$, $B_{2R_0/3}^\pm\setminus B_{R_0/2}^\pm$ and $B_{R_0/2}^\pm\setminus B_{r/2}^\pm$, and we observe that the integrals over $B_{R_0/2}^\pm\setminus B_{r/2}^\pm$ can be absorbed by the left hand side. Recalling \eqref{eq:stima_rho}, for $R_0\leq R_2=\min\{R_1,2(2CM_1^2)^{-1}\}$ we obtain

\begin{gather}
    \label{eq:35.1}
	\tau^4r^2\int_{B_{R_0}}\rho^{-2-2\tau}|\xi \overline{u}|^2+\sum_{k=0}^3 \tau^{6-2k}\int_{B_{R_0/2}^+ \setminus B_{r/2}^+}
	\rho^{2k+1-2\tau}|D^k u|^2+\\ \nonumber
	+\sum_{k=0}^3 \tau^{6-2k}\int_{B_{R_0/2}^- \setminus B_{r/2}^-}
	\rho^{2k+1-2\tau}|D^k w|^2\leq C\tau^2\overline{M}^2_1(J_0+J_1),
\end{gather}
for $\tau\geq \widetilde{\tau}$, with $C$ an absolute constant.

Let us estimate $J_0$ and $J_1$. {}From \eqref{eq:27.1} and recalling \eqref{eq:stima_rho}, we have
\begin{multline}
    \label{eq:36.1}
	J_0\leq\left(\frac{r}{4}\right)^{8-2\tau}\left\{
	\int_{B^+_{r/2}}\sum_{k=0}^3(r^{k-4}|D^k u|)^2+
	\int_{B^-_{r/2}}\sum_{k=0}^3(r^{k-4}|D^k w|)^2
	\right\}.
\end{multline}
By \eqref{eq:16.2}, we have that, for $(x,y)\in B^-_{r/2}$ and $k=0,1,2,3$,
\begin{equation}
    \label{eq:36.1bis}
	|D^k w|\leq C\sum_{h=k}^{2+k}r^{h-k}|(D^h u)(x,-y)|.
\end{equation}
By \eqref{eq:36.1}--\eqref{eq:36.1bis},
by making the change of variables
$(x,y)\rightarrow(x,-y)$ in the integrals involving the function $u(x,-y)$ and by using Lemma \ref{lem:intermezzo}, we get
\begin{equation}
    \label{eq:36.2}
	J_0\leq C\left(\frac{r}{4}\right)^{8-2\tau}
	\sum_{k=0}^5 r^{2k-8}\int_{B^+_{r/2}}|D^k u|^2
	 \leq C\left(\frac{r}{4}\right)^{-2\tau}\int_{B_r^+}|u|^2,
\end{equation}
where $C$ is an absolute constant. Analogously, we obtain
\begin{equation}
    \label{eq:37.1}
     J_1
	 \leq C\left(\frac{R_0}{2}\right)^{-2\tau}\int_{B_{R_0}^+}|u|^2.
\end{equation}
Recalling that $r<R<\frac{R_0}{2}$, by \eqref{eq:35.1}, \eqref{eq:36.2} and \eqref{eq:37.1}, it follows that
\begin{gather*}
    2^{-2}\tau^4(2r)^{-2\tau}\int_{B^+_{2r}\setminus B^+_{r/2} }|u|^2+\tau^{6}R^{1-2\tau}\int_{B_{R}^+ \setminus B_{r/2}^+}
	|u|^2\leq \\ \nonumber
	\leq \tau^4r^2\int_{B_{R_0}}\rho^{-2-2\tau}|\xi \overline{u}|^2+\sum_{k=0}^3\tau^{6-2k}\int_{B_{R_0/2}^+ \setminus B_{r/2}^+}
	\rho^{2k+1-2\tau}|D^ku|^2\leq\\ \nonumber
	\leq C\tau^2 \overline{M}^2_1\left[\left(\frac{r}{4}\right)^{-2\tau}\int_{B_r^+}|u|^2+
	\left(\frac{R_0}{2}\right)^{-2\tau}\int_{B_{R_0}^+}|u|^2
	\right],
\end{gather*}
for $\tau\geq \widetilde{\tau}$, with $C$ an absolute constant. Hence, we have

\begin{gather}
    \label{10.6.9}
	(2r)^{-2\tau}\int_{B^+_{2r}\setminus B^+_{r/2} }|u|^2+R^{1-2\tau}\int_{B_{R}^+ \setminus B_{r/2}^+}
	|u|^2\leq\\ \nonumber
	\leq C \overline{M}^2_1\left[\left(\frac{r}{4}\right)^{-2\tau}\int_{B_r^+}|u|^2+
	\left(\frac{R_0}{2}\right)^{-2\tau}\int_{B_{R_0}^+}|u|^2
	\right],
\end{gather}
Now, adding $R(2r)^{-2\tau}\int_{B^+_{r/2} }|u|^2$ to both sides of \eqref{10.6.9} we get the wished estimate \eqref{eq:37.1bis} for $r<R/2$ and $R<\overline{R}_0$, with $\overline{R}_0=R_2$.
\end{proof}

\begin{proof}[Proof of Theorem \ref{theo:40.teo}.]

Let us fix $R=\frac{\overline{R}_0}{4}$ in \eqref{eq:37.1bis} obtaining

\begin{gather}
    \label{eq:10.6.952}
    \frac{\overline{R}_0}{4} (2r)^{-2\tau}\int_{B^+_{2r}}|u|^2+\left(\frac{\overline{R}_0}{4}\right)^{1-2\tau}\int_{B_{\overline{R}_0/4}^+}
	|u|^2\leq\\ \nonumber
	\leq C \overline{M}^2_1\left[\left(\frac{r}{4}\right)^{-2\tau}\int_{B_r^+}|u|^2+
	\left(\frac{\overline{R}_0}{2}\right)^{-2\tau}\int_{B_{\overline{R}_0}^+}|u|^2
	\right],
\end{gather}
for every $\tau\geq \overline{\tau}$, with $\overline{\tau}, C$ absolute constants.

Now, choosing $\tau=\tau_0$, where

\begin{equation}
    \label{eq:10.6.1005}
    \tau_0=\overline{\tau}+\log_4\frac{4C \overline{M}^2_1\overline{N}}{\overline{R}_0}
\end{equation}
and

\begin{equation}
    \label{eq:10.6.1008}
    \overline{N}=\frac{\int_{B_{\overline{R}_0}^+}|u|^2}{\int_{B_{\overline{R}_0/4}^+}
	|u|^2}
\end{equation}
we have
$$\left(\frac{\overline{R}_0}{4}\right)^{1-2\tau}\int_{B_{R}^+}|u|^2\geq
C \overline{M}^2_1\left(\frac{\overline{R}_0}{2}\right)^{-2\tau}\int_{B_{\overline{R}_0}^+}|u|^2.$$
Hence, by \eqref{eq:10.6.952}, we obtain
\begin{equation}
    \label{eq:10.6.1014}
    \frac{\overline{R}_0}{4}(2r)^{-2\tau_0}\int_{B^+_{2r}}|u|^2
	\leq C \overline{M}^2_1\left(\frac{r}{4}\right)^{-2\tau_0}\int_{B_r^+}|u|^2,
\end{equation}
where $C$ is an absolute constant.
Using\eqref{eq:10.6.1008} and \eqref{eq:10.6.1014}, we have

\begin{equation}
    \label{eq:10.6.1024}
    \int_{B^+_{2r}}|u|^2
	\leq C\overline{N}^3\int_{B_r^+}|u|^2,
\end{equation}
where $C$ depends on $M_1$ only.

Now, let $r<s<\frac{\overline{R}_0}{16}$ and let $j=\left[\log_2\left(sr^{-1}\right)\right]$ (for $a\in \mathbb{R}^+$, $[a]$ denotes the integer part of $a$). We have $$2^jr\leq s <2^{j+1}r$$ and applying iteratively \eqref{eq:10.6.1024} we obtain

\begin{equation*}
    \int_{B^+_{s}}|u|^2\leq \int_{B^+_{2^{j+1}r}}|u|^2
	\leq \left(C\overline{N}^3\right)^{j+1}\int_{B_r^+}|u|^2\leq C\overline{N}^3\left(\frac{s}{r}\right)^{\log_2(C\overline{N}^3)}\int_{B_r^+}|u|^2.
\end{equation*}
Finally, coming back to the original coordinates and using Proposition \ref{prop:conf_map}, we can choose $s=\frac{2Kr}{r_0} (<\frac{\overline{R}_0}{16})$ in the above inequality and derive \eqref{eq:10.6.1102}--\eqref{eq:10.6.1108}, with $C=\frac{32K}{\overline{R}_0}$.
\end{proof}

\begin{cor}\label{application}
 Assume the same hypotheses of Theorem \ref{theo:40.teo} and let $E$ be a measurable subset of $\Gamma_{r_0}$ with positive 1-dim measure. We have that if
\begin{equation}\label{application-1}
 \begin{cases}
 Lv=0, \mbox{ in } \Omega_{r_0}, \\
v =\frac{\partial v}{\partial n}=|D^2v|=0, \mbox{ on } \Gamma_{r_0},   \\
D^3v=0, \mbox{ on } E,
\end{cases}
\end{equation}
then $$v\equiv 0, \quad \mbox{in } \Omega_{r_0}.$$
\end{cor}

\begin{proof}
We only sketch the proof and, without loss of generality, let us assume that $\Gamma_{r_0}$ is the interval $\mathcal{I}_{r_0}=(-r_0,r_0)$ in the $x$-axis. Also, for any point $P\in \mathcal{I}_{r_0}=(-r_0,r_0)$ we denote by $\mathcal{I}_{r}(P)$ the interval $(P-r,P+r)$, by $\mathcal{I}_{r}=\mathcal{I}_{r}(0)$.

It is enough to prove that $|D^3v|^2_{|\mathcal{I}_{r_0}}$ is an $A_p$ weight. In fact, by this property we have that $|D^3v|=0$ on $\Gamma_{r_0}$ (see \cite{l:GaLi}) and, by the uniqueness for Cauchy problem (see \cite[Section 3]{l:mrv11}), it follows that $v=0$ in $\Omega_{r_0}$. In order to prove that $|D^3v|^2_{|\mathcal{I}_{r_0}}$ is an $A_p$ weight, in view of the results in \cite{l:CoiF}, it is sufficient to prove that it satisfies a reverse H\"{o}lder inequality.

We can rewrite the doubling inequality \eqref{eq:10.6.1102} as follows
\begin{equation}
    \label{eq:new-doub}
    \int_{B^+_{2r}(P)}|v|^2\leq C_0\int_{B^+_{r}(P)}|v|^2, \mbox{ for every } P\in \mathcal{I}_{r_0/2}, \mbox{ and } r\leq r_0/C,
\end{equation}
where $C>2$ only depends on $\alpha_0$, $\gamma_0$, $\Lambda_0$, $M_0$, $\alpha$ and $C_0$ only depends on $\alpha_0$, $\gamma_0$, $\Lambda_0$, $M_0$, $\alpha$ and $v$, but is independent of $r$ and $P$ (the latter can be achieved by standard argument, see for instance \cite[Proposition 2.1]{l:DLMRVW}).

By the stability estimate for Cauchy problem for equation $Lv=0$ (\cite[Section 3]{l:mrv11}) we have that, for any $P\in \mathcal{I}_{r_0/2}$ and any $r\leq r_0/4C$,

\begin{equation}
    \label{eq:Cauchy}
    \int_{B^+_{r}(P)}|v|^2\leq C_1\left(r^7\int_{\mathcal{I}_{2r}(P)}|D^3v|^2\right)^{\delta}\left(\int_{B^+_{4r}(P)}|v|^2\right)^{1-\delta}.
\end{equation}
where $\delta\in (0,1)$ and  $C_1$ depend on $\alpha_0$, $\gamma_0$, $\Lambda_0$, $M_0$, $\alpha$. By \eqref{eq:new-doub} and \eqref{eq:Cauchy} we have
\begin{gather}
    \label{eq:1-appl}
    \int_{B^+_{4r}(P)}|v|^2\leq C^2_0\int_{B^+_{r}(P)}|v|^2 \leq\\\nonumber \leq C^2_0C_1\left(r^7\int_{\mathcal{I}_{2r}(P)}|D^3v|^2\right)^{\delta}\left(\int_{B^+_{4r}(P)}|v|^2\right)^{1-\delta},
\end{gather}
hence
\begin{gather}
    \label{eq:2-appl}
    \int_{B^+_{4r}(P)}|v|^2 \leq (C^2_0C_1)^{1/\delta}r^7\int_{\mathcal{I}_{2r}(P)}|D^3v|^2.
\end{gather}

Since $v\in H^4(B^+_{r_0})$ we have that $|D^3v|_{|\mathcal{I}_{4r}(P)}\in H^{1/2}(\mathcal{I}_{4r}(P))$ and by the imbedding theorem we have $|D^3v|\in L^{q}(\mathcal{I}_{4r}(P))$ for every $q\in (0,+\infty)$, see for instance \cite{l:adams}. Let us fix $q>2$. By imbedding estimates, standard trace inequalities, \eqref{eq:new-doub}, \eqref{eq:2-appl} and by Lemma \ref{lem:intermezzo} we have

\begin{gather*}
    \label{eq:3-appl}
    r^{3}\left(\dashint_{\mathcal{I}_{2r}(P)}|D^3v|^q\right)^{\frac{1}{q}} \leq C \left(r^8 \dashint_{B^+_{3r}(P)}|D^4v|^2 +r^6 \dashint_{B^+_{3r}(P)}|D^3v|^2 \right)^{\frac{1}{2}}\leq\\\nonumber
    \leq C \left(\dashint_{B^+_{4r}(P)}|v|^2\right)^{\frac{1}{2}}\leq C (C^2_0C_1)^{\frac{1}{2\delta}} r^3\left(\dashint_{\mathcal{I}_{2r}(P)}|D^3v|^2\right)^{\frac{1}{2}},
    \end{gather*}
hence we have proved the following reverse H\"{o}lder inequality
\begin{gather}
    \label{eq:3-appl}
    \left(\dashint_{\mathcal{I}_{2r}(P)}|D^3v|^q\right)^{\frac{1}{q}} \leq C (C^2_0C_1)^{\frac{1}{2\delta}} \left(\dashint_{\mathcal{I}_{2r}(P)}|D^3v|^2\right)^{\frac{1}{q}},
\end{gather}
which completes the proof.
\end{proof}

\section{Appendix} \label{sec:appendix}

In this Appendix we prove Carleman estimate \eqref{eq:24.4}. We proceed, similarly to \cite{l:CoKo}, \cite{l:mrv07}, \cite{l:Zhu}, in a standard way by iterating a suitable Carleman estimate for the Laplace operator.

In the present section we denote by $x_1, x_2$ the cartesian coordinate of a point  $x\in \mathbb{R}^2$.

\begin{prop}[{\bf Carleman estimate for $\Delta$}]\label{prop:Carlm-delta}
Let $r\in[0,1)$ and let $\varepsilon\in(0,1)$. Let us define
\begin{equation}
    \label{peso}
		\rho(x) = \phi_{\varepsilon}\left(|x|\right), \mbox{ for } x\in B_1\setminus \{0\},
\end{equation}
where
\begin{equation}
    \label{eq:24.3-delta}
		\phi_{\varepsilon}(s) = \frac{s}{\left(1+s^{\varepsilon}\right)^{1/\varepsilon}}.
\end{equation}
Then there exist $\tau_0>1$, $C>1$, only depending on $\epsilon$, such that
\begin{gather}\label{Carlm-delta}
	\tau^2r\int\rho^{-1-2\tau}u^2dx+\sum_{k=0}^1 \tau^{3-2k}\int\rho^{2k+\epsilon-2\tau}|D^ku|^2dx\leq \\\nonumber
    \leq C\int\rho^{4-2\tau}|\Delta u|^2dx,
\end{gather}
for every $\tau\geq \tau_0$ and for every $u\in C^\infty_0(B_1\setminus \overline{B}_{r/4})$.
\end{prop}

\begin{proof}
Let $u$ be an arbitrary function in $C^{\infty}_0\left(B_1\setminus \overline{B}_{r/4}\right)$ and let us express the two dimensional Laplacian in polar coordinates $(\varrho,\vartheta)$, that is (here and in the sequel $\mathbb{S}^1=\partial B_1$)

\begin{equation}\label{laplace-polar-1}
    \Delta u=u_{\varrho\varrho}+\frac{1}{\varrho}u_{\varrho}+\frac{1}{\varrho^2}u_{\vartheta\vartheta}, \mbox{ for } \varrho>0, \vartheta\in \mathbb{S}^1.
\end{equation}
By the change of variable $\varrho=e^t$, $\widetilde{u}(t,\vartheta)=u\left(e^t,\vartheta\right)$, $(t,\vartheta)\in (-\infty,0)\times \mathbb{S}^1$ we have

\begin{equation}\label{laplace-polar-2}
    e^{2t}(\Delta u)(e^t,\vartheta)=\mathcal{L}\widetilde{u}:=(\widetilde{u}_{tt}+\widetilde{u}_{\vartheta\vartheta})(t,\vartheta) , \mbox{ for } (t,\vartheta)\in (-\infty,0)\times \mathbb{S}^1.
\end{equation}
For sake of brevity, for any smooth function $h$, we shall write $h'$, $h''$, ... instead of $h_t$, $h_{tt}$, ... By \eqref{peso} we have (here and in the sequel we omit the subscript $\varepsilon$)

\begin{equation}\label{peso-1}
\varphi(t):=\log(\phi(e^t))=t-\varepsilon^{-1}\log\left(1+e^{\varepsilon t}\right), \quad \mbox{ for } t\in(-\infty,0).
\end{equation}

We have
\begin{equation}\label{peso-derivate}
\varphi'(t)=\frac{1}{1+e^{\varepsilon t}}, \quad  \varphi''(t)=-\frac{\varepsilon e^{\varepsilon t}}{(1+e^{\varepsilon t})^2} \mbox{ ,}  \quad \mbox{ for } t\in(-\infty,0).
\end{equation}

Let

\begin{equation*}
f(t,\vartheta)=e^{-\tau\varphi}\widetilde{u}(t,\vartheta) , \mbox{ for } (t,\vartheta)\in(-\infty,0)\times \mathbb{S}^1.
\end{equation*}
We have

\begin{equation}\label{L-tau}
\mathcal{L}_{\tau}f:=e^{-\tau\varphi}\mathcal{L}(e^{\tau\varphi}f)=\underset{\mathcal{A}_{\tau}f}{\underbrace{\tau \varphi''f+2\tau\varphi'f'}}+\underset{\mathcal{S}_{\tau}f}{\underbrace{\tau^2\varphi'^2f+f''+f_{\vartheta\vartheta}}}.
\end{equation}

Denote by $\int(\cdot)$ the integral $\int^{0}_{-\infty}\int_{\mathbb{S}^1}(\cdot) d\vartheta dt$ and let

\begin{equation}\label{gamma}
\gamma:=\frac{1}{\varphi'}=1+e^{\varepsilon t}.
\end{equation}

We have

\begin{gather}\label{square}
\int\gamma\left\vert \mathcal{L}_{\tau}f\right\vert^2\geq \int\gamma\left\vert\mathcal{A}_{\tau}f\right\vert^2 +2\int\gamma\mathcal{A}_{\tau}f\mathcal{S}_{\tau}f,
\end{gather}

\begin{gather}\label{commutator}
2\int\gamma\mathcal{A}_{\tau}f\mathcal{S}_{\tau}f=2\int\gamma\left(\tau \varphi''f+2\tau\varphi'f'\right)f_{\vartheta\vartheta}+\\ \nonumber
+2\int\gamma\left(\tau \varphi''f+2\tau\varphi'f'\right)\left(\tau^2\varphi'^2f+f''\right):=I_1+I_2.
\end{gather}

Let us examine $I_1$.

By integration by parts and taking into account \eqref{gamma}, we have

\begin{gather*}
I_1=2\int\left(\tau\gamma\varphi''ff_{\vartheta\vartheta}+2\tau \gamma\varphi'f'f_{\vartheta\vartheta}\right)=\\
=2\int\left(-\tau\gamma\varphi'' f_{\vartheta}^2-2\tau \gamma\varphi'f'_{\vartheta}f_{\vartheta}\right)=2\int\left(-\tau\gamma\varphi'' f_{\vartheta}^2-
\tau \gamma\varphi'\left( f_{\vartheta}^2\right)^{\prime}\right)=\\
=2\tau\int\gamma'\varphi' f_{\vartheta}^2=2\varepsilon\tau\int\frac{e^{\varepsilon t}}{1+e^{\varepsilon t}} f_{\vartheta}^2.
\end{gather*}
Hence, we have
\begin{gather}\label{I-1}
I_1=2\varepsilon\tau\int\frac{e^{\varepsilon t}}{1+e^{\varepsilon t}}f_{\vartheta}^2.
\end{gather}
Now, let us consider $I_2$.

By integration by parts, we have

\begin{gather}\label{I-2}
I_2=2\int\gamma\left(\tau^3 \varphi''\varphi'^2f^2+2\tau^3\varphi'^3ff'+\tau\varphi''ff''+2\tau \varphi'f'f''\right)=\\ \nonumber
=2\int\tau^3\gamma \varphi''\varphi'^2f^2+\tau^3\left(\gamma\varphi'^3\right)\left(f^2\right)^{\prime}-\tau\left(\gamma\varphi''f\right)^{\prime}f'+
\tau\gamma\varphi^{\prime}\left(f'^2\right)^{\prime}.
\end{gather}
Since $\gamma\varphi^{\prime}=1$, we have
\begin{equation}\label{I-2-2}
\int\gamma\varphi^{\prime}\left(f'^2\right)^{\prime}=\int\left(f'^2\right)^{\prime}=0,
\end{equation}
and the last term in the last integral of \eqref{I-2} vanishes.
By considering the first and second term in the last integral of \eqref{I-2}, we have
\begin{equation*}
2\int\tau^3\gamma \varphi''\varphi'^2f^2+\tau^3\left(\gamma\varphi'^3\right)\left(f^2\right)^{\prime}=2\int\tau^3 \varphi''\varphi'f^2-\tau^3\left(\varphi'^2\right)^{\prime}f^2=-2\tau^3\int\varphi''\varphi'f^2.
\end{equation*}
By \eqref{peso-derivate}, we have

$$\varphi''\varphi'=-\frac{\varepsilon e^{\varepsilon t}}{\left(1+e^{\varepsilon t}\right)^3},$$
and, therefore,

\begin{equation}\label{I-2-3}
2\int\tau^3\gamma \varphi''\varphi'^2f^2+\tau^3\left(\gamma\varphi'^3\right)\left(f^2\right)^{\prime}=2\tau^3\int\frac{\varepsilon e^{\varepsilon t}}{\left(1+e^{\varepsilon t}\right)^3}f^2.
\end{equation}
Concerning the third term in the last integral of \eqref{I-2}, we have

\begin{gather*}
2\int-\tau\left(\gamma\varphi''f\right)^{\prime}f'=2\tau\int-\gamma\varphi''f'^2-\left(\gamma\varphi''\right)^{\prime}ff'
=2\tau\int-\gamma\varphi''f'^2-\frac{1}{2}\left(\gamma\varphi''\right)''f^2,
\end{gather*}
and, by \eqref{gamma}, \eqref{peso-derivate}, we have
$$-\gamma\varphi''=\frac{\varepsilon e^{\varepsilon t}}{1+e^{\varepsilon t}}.$$
In addition, it is easy to check that
$$\left\vert\left(\gamma\varphi''\right)''\right\vert\leq \frac{\varepsilon^3e^{\varepsilon t}}{\left(1+e^{\varepsilon t}\right)^3}, \quad \mbox{ for every } t\in (-\infty,0),$$ hence

\begin{equation}\label{I-2-4}
2\int-\tau\left(\gamma\varphi''f\right)^{\prime}f'\geq 2\tau\int\frac{\varepsilon e^{\varepsilon t}}{1+e^{\varepsilon t}}f'^2-\tau\int \frac{\varepsilon^3e^{\varepsilon t}}{\left(1+e^{\varepsilon t}\right)^3} f^2.
\end{equation}
By using inequalities \eqref{I-2}-\eqref{I-2-4}, we have
\begin{gather}\label{I-2-5}
I_2\geq 2\tau^3\int\frac{\varepsilon e^{\varepsilon t}}{\left(1+e^{\varepsilon t}\right)^3}\left(1-\varepsilon^2\tau^{-2}\right)f^2+2\tau\int\frac{\varepsilon e^{\varepsilon t}}{1+e^{\varepsilon t}}f'^2\geq\\\nonumber
\geq \tau^3\int\frac{\varepsilon e^{\varepsilon t}}{\left(1+e^{\varepsilon t}\right)^3}f^2+2\tau\int\frac{\varepsilon e^{\varepsilon t}}{1+e^{\varepsilon t}}f'^2,
\end{gather}
for every $\tau\geq \varepsilon/\sqrt{2}$.

By \eqref{commutator}--\eqref{I-2-5} we have

\begin{gather}\label{commutator-final}
\int\gamma\left\vert L_{\tau}f\right\vert^2\geq \int\gamma\left\vert\mathcal{A}_{\tau}f\right\vert^2+\\ \nonumber
+2\varepsilon\tau\int\frac{e^{\varepsilon t}}{1+e^{\varepsilon t}}\left(f'^2+f_{\vartheta}^2\right)+\tau^3\varepsilon\int\frac{ e^{\varepsilon t}}{\left(1+e^{\varepsilon t}\right)^3}f^2
\end{gather}
for every $\tau\geq \varepsilon/\sqrt{2}$ and for every $f\in C_0^{\infty}((-\infty,0)\times\mathbb{S}^1)$.

In order to obtain the first term on the left hand side of \eqref{Carlm-delta}, inspired by \cite[Theorem 2.1]{l:Bakri1}, we use the first term on the right hand side of \eqref{commutator-final}.

Observe that by the trivial inequality $(a+b)^2\geq\frac{1}{2}a^2-b^2$ and by \eqref{peso-derivate}, \eqref{gamma}, we get

\begin{gather}\label{antisimm-1}
\int\gamma\left\vert\mathcal{A}_{\tau}f\right\vert^2\geq \frac{1}{2}\int\gamma \left(2\tau\varphi'f'\right)^2-\int\gamma\left(\tau \varphi''f\right)^2= \\ \nonumber
=2\tau^2\int \frac{1}{1+e^{\varepsilon t}}f'^2- \varepsilon^2\tau^2\int\frac{e^{2\varepsilon t}}{(1+e^{\varepsilon t})^3}f^2, \mbox{ for every } \tau\geq \frac{\varepsilon}{\sqrt{2}}.
\end{gather}
By inserting the inequality \eqref{antisimm-1} in \eqref{commutator-final} we have

\begin{gather}\label{antisimm-2}
\int\gamma\left\vert \mathcal{L}_{\tau}f\right\vert^2\geq 2\tau^2\int \frac{1}{1+e^{\varepsilon t}}f'^2+ \varepsilon\tau^3\int \frac{e^{\varepsilon t}\left(1-\varepsilon\tau^{-1}e^{\varepsilon t}\right)}{(1+e^{\varepsilon t})^3}f^2+\\ \nonumber
+2\varepsilon\tau\int\frac{e^{\varepsilon t}}{1+e^{\varepsilon t}}\left(f'^2+f_{\vartheta}^2\right), \mbox{ for every } \tau\geq\frac{\varepsilon}{\sqrt{2}}.
\end{gather}
Now, noticing that $\left(1-\varepsilon\tau^{-1}e^{\varepsilon t}\right)\geq 1/2$ for every $\tau\geq \varepsilon/2$ and by using the trivial estimate $\frac{1}{1+e^{\varepsilon t}}\geq 1/2$ for $t\in (-\infty,0)$, \eqref{antisimm-2} gives

\begin{gather}\label{antisimm-3}
\int\gamma\left\vert \mathcal{L}_{\tau}f\right\vert^2\geq \tau^2\int f'^2+ \frac{\varepsilon\tau^3}{8}\int e^{\varepsilon t}f^2
+\varepsilon\tau\int e^{\varepsilon t}\left(f'^2+f_{\vartheta}^2\right),
\end{gather}
for every $\tau\geq \frac{\varepsilon}{\sqrt{2}}$.

Now, by Proposition \ref{prop:Hardy} we have
\begin{gather}\label{new-Hardy}
\int^0_{-\infty} f^2(t,\vartheta)e^{-t}dt=\int^1_{0} s^{-2}f^2(\log s,\vartheta)ds\leq \\\nonumber
\leq 4\int^1_{0} \left\vert\frac{\partial}{\partial s}f(\log s,\vartheta)\right\vert^2 ds=4\int^0_{-\infty} f'^2(t,\vartheta)e^{-t}dt, \quad \hbox{ for every } \vartheta\in \mathbb{S}^1.
\end{gather}

On the other side, since $f(t,\vartheta)=0$ for every $t\leq \log (r/4)$, by \eqref{new-Hardy} we have

$$\int^0_{-\infty} f^2(t,\vartheta)e^{-t}dt\leq 4\int^{\log\frac{r}{4}}_{-\infty} f'^2(t,\vartheta)e^{-t}dt\leq\frac{16}{r}\int^{0}_{-\infty} f'^2(t,\vartheta), \quad \hbox{ for every } \vartheta\in \mathbb{S}^1.$$
By integrating over $\mathbb{S}^1$ the above inequality and by using \eqref{antisimm-3}, we have

\begin{equation}\label{antisimm-4}
\int f^2e^{-t}\leq\frac{16}{r}\int f'^2\leq \frac{16}{\tau^2r}\int\gamma\left\vert \mathcal{L}_{\tau}f\right\vert^2, \mbox{ for every } \tau\geq \frac{\varepsilon}{\sqrt{2}}.
\end{equation}
By \eqref{antisimm-4} and \eqref{antisimm-3} we have
\begin{gather}\label{antisimm-5}
C\int\left\vert \mathcal{L}_{\tau}f\right\vert^2\geq \varepsilon\tau^3\int e^{\varepsilon t}f^2+\\\nonumber
+\varepsilon\tau\int e^{\varepsilon t}\left(f'^2+f_{\vartheta}^2\right)+\tau^2r\int f^2e^{-t}, \mbox{ for every } \tau\geq \frac{\varepsilon}{\sqrt{2}},
\end{gather}
where $C$ is an absolute constant.

Now we come back to the original coordinates. Recalling that $f(t,\vartheta)=e^{-\tau\varphi}u(e^t,\vartheta)$, and by using \eqref{peso}, \eqref{laplace-polar-2} and \eqref{L-tau}, we have

\begin{gather}\label{final-1}
\int^{0}_{-\infty}\int_{\mathbb{S}^1}\left\vert \mathcal{L}_{\tau}f\right\vert^2d\vartheta dt=\int^{0}_{-\infty}\int_{\mathbb{S}^1}e^{-2\tau\varphi(t)}e^{4t}|(\Delta u)(e^t,\vartheta)|^2d\vartheta dt=\\\nonumber
= \int^{1}_{0}\int_{\mathbb{S}^1}e^{-2\tau\varphi(\log\varrho)}\varrho^{3}|(\Delta u)(e^t,\vartheta)|^2d\vartheta d\varrho=\int_{B_1}\rho^{-2\tau}|x|^2|\Delta u|^2dx.
\end{gather}
Similarly, we have
\begin{equation}\label{final-2}
\int^{0}_{-\infty}\int_{\mathbb{S}^1} f^2e^{-t}d\vartheta dt=\int_{B_1}\rho^{-2\tau}|x|^{-3} u^2dx
\end{equation}
and

\begin{equation}\label{final-3}
\int^{0}_{-\infty}\int_{\mathbb{S}^1} f^2e^{\varepsilon t}d\vartheta dt=\int_{B_1}\rho^{-2\tau}|x|^{\varepsilon-2} u^2dx.
\end{equation}
Concerning the second integral on the right hand side of \eqref{antisimm-3}, let $\delta\in (0,1)$ to be choosen later, we have

\begin{gather}\label{final-4}
\int^{0}_{-\infty}\int_{\mathbb{S}^1} e^{\varepsilon t}\left(f'^2+f_{\vartheta}^2\right)d\vartheta dt\geq\delta\int^{0}_{-\infty}\int_{\mathbb{S}^1} e^{\varepsilon t}\left(f'^2+f_{\vartheta}^2\right)d\vartheta dt\geq\\\nonumber
\geq \frac{\delta}{2}\int^{0}_{-\infty}\int_{\mathbb{S}^1}e^{\varepsilon t}e^{-2\tau\varphi(t)}\left(|u_{\varrho}(e^t,\vartheta)|^2e^{2t}
+|u_{\vartheta}(e^t,\vartheta)|^2-
2\tau^2|u(e^t,\vartheta)|^2\right)d\vartheta dt=\\\nonumber
=\frac{\delta}{2}\int_{B_1}\rho^{-2\tau}|x|^{\varepsilon-2} \left(|x|^{2}|\nabla u|^2-2\tau^{2}|u|^2\right)dx.
\end{gather}
Choosing $\delta=\frac{1}{2}$, and by \eqref{antisimm-5} and \eqref{final-1}--\eqref{final-4}, we have

\begin{gather}\label{final-5}
C\int_{B_1}\rho^{-2\tau}|x|^2|\Delta u|^2dx\geq \frac{\varepsilon\tau}{4}\int_{B_1}\rho^{-2\tau}|x|^{\varepsilon} |\nabla u|^2dx+\\\nonumber
+\frac{\varepsilon\tau^3}{2}\int_{B_1}\rho^{-2\tau}|x|^{\varepsilon-2} u^2dx+\tau\int_{B_1}\rho^{-2\tau}|x|^{-3} u^2dx,
\end{gather}
for every $\tau\geq \frac{\varepsilon}{\sqrt{2}}$ and for every $u\in C^\infty_0(B_1\setminus \overline{B}_{r/4})$.
Finally, since by \eqref{peso} we have $$2^{-\frac{1}{\varepsilon}}|x|\leq \rho(x)\leq |x|,$$
we can replace $\tau$ in \eqref{final-5} by $(\tau-1)$ and we obtain the desired inequality \eqref{Carlm-delta}.
\end{proof}

\medskip

In order to prove Proposition \ref{prop:Carleman}, we need the following

\begin{lem}\label{formule}
Given $\zeta\in C^2(B_1\setminus\{0\})$ and $u\in
C^\infty_0(B_1\setminus\{0\})$, the following identities hold
true:
\begin{subequations}
\label{identity}
\begin{equation}
\label{identity1} \int \zeta u\Delta u=-\int(\zeta|\nabla u|^2+(\nabla
u\cdot\nabla \zeta)u),
\end{equation}
\begin{equation}
\label{identity2} \int \zeta\sum_{j,k=1}^2|\partial_{jk}u|^2=\int(-D^2\zeta\nabla u\cdot
\nabla u+\Delta \zeta|\nabla u|^2+\zeta(\Delta u)^2),
\end{equation}
\begin{gather}
\label{identity3}
\int \zeta\sum_{i,j,k=1}^2|\partial_{ijk}u|^2=-\int \zeta\Delta u \Delta^2
u+\\\nonumber
\int(-\hbox{tr}(D^2 uD^2 \zeta D^2u)
+\Delta \zeta|D^2 u|^2+\frac{1}{2}\Delta \zeta(\Delta u)^2).
\end{gather}
\end{subequations}
\end{lem}

\begin{proof}
Concerning \eqref{identity1} it is enough to note that
\begin{equation}
\label{proof_identity1}
\int \zeta u\Delta u=-\int\nabla
u\cdot\nabla(\zeta u)=-\int(\zeta|\nabla u|^2+(\nabla u\cdot\nabla \zeta)u).
\end{equation}
In order to prove \eqref{identity2}, let us compute
\begin{gather*}
    \label{proof_identity2}
\int \zeta(\Delta u)^2=\int \sum_{j,k=1}^2 \zeta \partial_{jj}u \partial_{kk}u=\\
-\int
\sum_{j,k=1}^2 (\partial_{k}\zeta \partial_{jj}u \partial_{k}u+\zeta \partial_{jjk}u \partial_{k}u)=
\int \sum_{j,k=1}^2 \partial_{j}(\partial_{k}\zeta \partial_{k}u)\partial_{j}u+\partial_{j}(\zeta\partial_{k}u)\partial_{jk}u=\\
=\int D^2 \zeta\nabla u\cdot\nabla u+\zeta\sum_{j,k=1}^2|\partial_{jk}u|^2+2\sum_{j,k=1}^2
\partial_{k}\zeta \partial_{jk}u\partial_{j}u.
\end{gather*}
Noticing that $\partial_{jk}u\partial_{j}u=\frac{1}{2}\partial_k(\partial_ju)^2$ and
integrating by parts the last term on the right hand side of the
above identity, we obtain \eqref{identity2}.

In order to derive
\eqref{identity3}, let us apply \eqref{identity1} to $\Delta u$,
obtaining
\begin{equation}
    \label{proof_identity3_1}
\int \zeta\Delta u\Delta^2 u=-\int(\zeta|\nabla \Delta u|^2+(\nabla \Delta
u\cdot\nabla \zeta)\Delta u).
\end{equation}
{}From \eqref{identity2}, we have
\begin{equation}
    \label{proof_identity3_2}
-\int \zeta|\nabla \Delta u|^2=-\int \hbox{tr}(D^2u D^2 \zeta
D^2 u)-\Delta \zeta |D^2 u|^2+\zeta\sum_{i,j,k=1}^2|\partial_{ijk}u|^2,
\end{equation}
and, in addition,
\begin{equation}
    \label{proof_identity3_3}
-\int (\nabla \Delta u\cdot\nabla \zeta)\Delta u=-\frac{1}{2}\int
\sum_{j=1}^2\partial_{j}\zeta\partial_{j}(\Delta u)^2=\frac{1}{2} \int \Delta \zeta(\Delta u)^2.
\end{equation}
{}From \eqref{proof_identity3_1}--\eqref{proof_identity3_3},
identity \eqref{identity3} follows.
\end{proof}

\begin{proof}[Proof of Proposition \ref{prop:Carleman}.]

Let $r\in (0,1)$. For the sake of brevity, given two quantities $X,Y$ in which the parameter $\tau$ in involved, we will write $X\lesssim Y$ to mean that there exist constants $C, C'$ independent on $\tau$ and $r$ such that $X\leq CY$ for every $\tau\geq C'$.

Let $U$ be an arbitrary function of $C^\infty_0(B_1\setminus \overline{B}_{r/4})$. By applying \eqref{Carlm-delta} to $u=\Delta U$ we have

\begin{gather}\label{Carlm-delta2-1}
	\int\rho^{8-2\tau}|\Delta^2 U|^2=\int\rho^{4-2(\tau-2)}|\Delta(\Delta U)|^2\gtrsim \\\nonumber
   \gtrsim \tau^2r\int\rho^{-1-2(\tau-2)}|\Delta U|^2=\tau^2r\int\rho^{4-2(\tau+\frac{1}{2})}|\Delta U|^2 \gtrsim\\\nonumber
   \gtrsim \tau^4r^2\int\rho^{-2-2\tau}|U|^2, \mbox{ for every }  U\in C^\infty_0(B_1\setminus \overline{B}_{r/4}).
\end{gather}

Similarly we have

\begin{gather}\label{Carlm-delta2-2}
	\int\rho^{8-2\tau}|\Delta^2 U|^2\gtrsim \tau^3\int\rho^{\varepsilon-2(\tau-2)}|\Delta U|^2 =\\\nonumber
   =\tau^3\int\rho^{4-2(\tau-\frac{\varepsilon}{2})}|\Delta U|^2 \gtrsim\\\nonumber
   \gtrsim \tau^6\int\rho^{2\varepsilon-2\tau}|U|^2+\tau^4\int\rho^{2+2\varepsilon-2\tau}|\nabla U|^2,
\end{gather}
hence, by \eqref{Carlm-delta2-1} and \eqref{Carlm-delta2-2}, we have

\begin{gather}\label{Carlm-delta2-3}
 \tau^4r^2\int\rho^{-2-2\tau}|U|^2+\tau^6\int\rho^{2\varepsilon-2\tau}|U|^2+\\\nonumber
 +\tau^4\int\rho^{2+2\varepsilon-2\tau}|\nabla U|^2\lesssim \int\rho^{8-2\tau}|\Delta^2 U|^2.
\end{gather}
Now we estimate {}from above the terms with \textit{second derivatives} of $U$.

Let us apply Lemma \ref{formule} with $\zeta=\zeta_1:=\rho^{4+2\varepsilon-2\tau}$. Since
\begin{equation}\label{stime-kappa-1}
 |\nabla \zeta_1| \lesssim \tau \rho^{3+2\varepsilon-2\tau},  \mbox{ and }\quad |D^2 \zeta_1| \lesssim \tau^2 \rho^{2+2\varepsilon-2\tau},
\end{equation}
by \eqref{identity2} and \eqref{stime-kappa-1} we get

\begin{gather}\label{Carlm-delta2-4}
 \int\rho^{4+2\varepsilon-2\tau}|D^2U|^2\lesssim \int\rho^{4+2\varepsilon-2\tau}|\Delta U|^2+\tau^2\int\rho^{2+2\varepsilon-2\tau}|\nabla U|^2.
\end{gather}

By \eqref{Carlm-delta} we have
\begin{gather}\label{Carlm-delta2-5}
 \int\rho^{4+2\varepsilon-2\tau}|\Delta U|^2=\int\rho^{\varepsilon-2(\tau-2-\frac{\varepsilon}{2})}|\Delta U|^2\lesssim \\\nonumber
 \lesssim\tau^{-3} \int\rho^{8+\varepsilon-2\tau}|\Delta^2 U|^2 \leq \tau^{-3} \int\rho^{8-2\tau}|\Delta^2 U|^2.
\end{gather}
Now, we can use \eqref{Carlm-delta2-3} to estimate the second integral on the right hand side of \eqref{Carlm-delta2-4}, obtaining
\begin{gather}\label{Carlm-delta2-6}
\tau^2\int\rho^{2+2\varepsilon-2\tau}|\nabla U|^2\lesssim  \tau^{-2}\int\rho^{8-2\tau}|\Delta^2 U|^2.
\end{gather}
By \eqref{Carlm-delta2-4}, \eqref{Carlm-delta2-5} and \eqref{Carlm-delta2-6} we have

\begin{gather}\label{Carlm-delta 2-5-1}
 \tau^{2} \int\rho^{4+2\varepsilon-2\tau}|D^2 U|^2\lesssim  \int\rho^{8-2\tau}|\Delta^2 U|^2.
\end{gather}

Let us estimate {}from above the terms with \textit{third derivatives} of $U$. To this aim, we apply Lemma \ref{formule} with $\zeta=\zeta_2:=\rho^{6+2\varepsilon-2\tau}$, and likewise to \eqref{stime-kappa-1} we have
\begin{equation}\label{stime-kappa-2}
 |\nabla \zeta_2| \lesssim \tau \rho^{5+2\varepsilon-2\tau},  \mbox{ and } \quad |D^2 \zeta_2| \lesssim \tau^2 \rho^{4+2\varepsilon-2\tau}.
\end{equation}
By \eqref{identity3} and \eqref{stime-kappa-2} we have

\begin{gather}\label{Carlm-delta2-6-1}
 \int\rho^{6+2\varepsilon-2\tau}|D^3 U|^2\lesssim  \int\rho^{6+2\varepsilon-2\tau}|\Delta U||\Delta^2 U|+\tau^2\int\rho^{4+2\varepsilon-2\tau}|D^2 U|^2.
\end{gather}
As next step, we estimate {}from above the first term on the right hand side of \eqref{Carlm-delta2-6-1} as follows
\begin{gather*}
 \int\rho^{6+2\varepsilon-2\tau}|\Delta U||\Delta^2 U|\leq \frac{1}{2}\int\rho^{4+2\varepsilon-2\tau}|\Delta U|^2+\frac{1}{2}\int\rho^{8+2\varepsilon-2\tau}|\Delta^2 U|^2.
\end{gather*}
The above inequality, \eqref{Carlm-delta 2-5-1} and \eqref{Carlm-delta2-6-1} give

\begin{gather}\label{Carlm-delta2-6-2}
 \int\rho^{6+2\varepsilon-2\tau}|D^3 U|^2\lesssim  \int\rho^{8-2\tau}|\Delta^2 U|^2.
\end{gather}
Summing up, \eqref{Carlm-delta2-3}, \eqref{Carlm-delta 2-5-1} and  \eqref{Carlm-delta2-6-2} we have
	\begin{equation}
    \label{5.Carleman_bilapl}
\tau^4r^2\int\rho^{-2-2\tau}|U|^2+\sum_{k=0}^3\tau^{6-2k}\int \rho^{2k+2\epsilon-2\tau}|D^k
U|^2\lesssim \int \rho^{8-2\tau}(\Delta^2 U)^2.
\end{equation}
Finally, choosing $\varepsilon=\frac{1}{2}$ in \eqref{5.Carleman_bilapl} we obtain the wished estimate \eqref{eq:24.4}.
\end{proof}

\bibliographystyle{plain}

\end{document}